\documentclass[11pt]{article}
\usepackage{latexsym}
\usepackage{amssymb,amsmath,amsthm}
\usepackage{addfont} 
\addfont{OT1}{rsfs10}{\rsfs}
\usepackage{graphicx}
\RequirePackage{tikz} 

\usetikzlibrary{intersections}

\usepackage{epsfig}
\usepackage{multicol}
\usepackage{pst-grad} 
\usepackage{pst-plot} 
\usepackage{pst-node} 
\usepackage{pst-tree}
\usetikzlibrary{shapes, arrows}
\usepackage{tkz-graph}
\usepackage{tkz-berge}
\usepackage[all]{xy}
\graphicspath{{fig/}}
\usepackage{subfig}
\usepackage{amsfonts}

\theoremstyle{definition}
\newtheorem{theorem}{Theorem}[section]
\newtheorem{lemma}[theorem]{Lemma}
\newtheorem{observation}[theorem]{Observation}

\newtheorem{remark}[theorem]{Remark}
\newtheorem{corollary}[theorem]{Corollary}

\begin{document}

\title{{\bf A Connection between Metric Dimension and Distinguishing Number of Graphs}}

\author{Meysam Korivand and Nasrin Soltankhah \thanks{Corresponding author}}

\date{}

\maketitle

\begin{center} 

Department of Mathematics, Faculty of Mathematical Sciences, \\
Alzahra University, Tehran, Iran \\
e-mail: {\tt mekorivand@gmail.com or m.korivand@alzahra.ac.ir,\\ soltan@alzahra.ac.ir}

\end{center}

\begin{abstract} 
In this paper, we introduce a connection between two classical concepts of graph theory: \; metric dimension and distinguishing number. 
For a given graph $G$, let ${\rm dim}(G)$ and $D(G)$ represent its metric dimension and distinguishing number, respectively.
We show that in connected graphs, any resolving set breaks the symmetry in the graphs. Precisely, if 
$G$ 
is a connected graph with a resolving set 
$S=\{v_1, v_2, \ldots, v_n \}$, 
then  
$\{\{v_1\}, \{v_2\}, \ldots, \{v_n\}, V(G)\setminus S \}$ 
is a partition of 
$V(G)$ 
into a distinguishing coloring, and as a consequence   
$D(G)\leq {\rm dim}(G)+1$. 
Furthermore, we construct graphs $G$ such that $D(G)=n$ and ${\rm dim}(G)=m$ for all values of $n$ and $m$, where $1\leq n< m$.
Using this connection, we have characterized all graphs 
$G$ 
of order 
$n$ 
with  
$D(G) \in \{n-1, n-2\}$.
For any graph $G$, let 
$G_c = G$ 
if 
$G$ 
is connected,
and 
$G_c = \overline{G}$ 
if 
$G$ 
is disconnected. 
Let 
$G^{\ast}$ 
denote the twin graph obtained from 
$G$ 
by contracting any maximal set of vertices with the same open or close neighborhood into a vertex. 
Let 
{\rsfs F} 
be the set of all graphs except graphs $G$ with the property that 
${\rm dim}(G_c)=|V(G)|-4$, 
${\rm diam}(G_c) \in \{2, 3\}$ 
and 
$5\leq |V(G_{c}^{\ast})| \leq 9$. 
We characterize all graphs
$G \in$ {\rsfs F} 
of order 
$n$
with the property that 
$D(G)= n-3$.
\end{abstract}


\medskip\noindent
{\bf AMS Subj.\ Class}: 05C15. 

\section{Introduction} 
For a connected graph $G$, a subset $S \subseteq V(G)$ is a {\it resolving set} if for any two vertices 
$g_1$  and $g_2$ of $G$, there exists vertex $s \in S$ such that ${\rm d}(g_1, s)\neq {\rm d}(g_2, s)$.
The smallest size that can be taken by a resolving set $S$ is called the {\it metric dimension} and is denoted by ${\rm dim}(G)$. 
The concept of resolving sets was introduced independently and simultaneously by Slater \cite{slater} and Harary \& Melter \cite{harary} in 1975-6. The concept of metric dimension is one of the most widely used and fundamental concepts in graph theory, and more than 1000 articles have been written on this concept. Today, after five decades, the metric dimension problems are still the topic of many researchers in graph theory. This concept in computer science is one of the classic examples of an NP-hard problem for education. Many uses of this concept have been found in other sciences such as computer science and chemistry. Various applications of this concept have been found so far to solve real-world problems such as
privacy in social networks, error correcting codes, locating intruders in networks, chemistry, robot navigation, pattern recognition, image processing and coin weighing \cite{2a, 1a, harary, 0, Khuller, 13, slater, 1.1}. For more information on the applications of the metric dimension, see the recently published survey \cite{survey2}.
Many versions of the metric dimension have been introduced and studied, including:  Strong Metric Dimension, $k$-Metric Dimension, Local Metric Dimension, Adjacency Dimension, Non-local Metric Dimension, Edge Metric Dimension, Resolving Partitions, Fractional Metric Dimension and Mixed Metric Dimension   
\cite{b.1, c.4, b.2, c.5, Klavžar0, c.7, c.6, c.9, c.8}. 
For complete information about the versions of the standard metric dimension introduced so far, see also the survey \cite{survey1}.

In 1977, Babai \cite{babai} presented a concept, which has been the infrastructure of many deep researches and definitions in graph theory and group theory. In this concept, the goal is to break symmetry in graphs by a vertex partition of graphs with smallest size. Since each symmetry in a graph represents a non-trivial automorphism of the graph and the set of automorphisms of a graph forms a group, the deep study of this concept has been related to the knowledge of group theory in many cases. Two decades later, when Albertson and Collins \cite{alber} studied this concept, it received  widespread attention. After that, this concept was added to the graph theory literature with the names of {\it distinguishing labeling} or {\it asymmetric coloring}. A distinguishing coloring of a graph is a vertex coloring such that there is no color preserving non-trivial automorphism of the graph. 
The minimum color required for a distinguishing coloring of a graph $G$ is indicated by $D(G)$ and is called {\it distinguishing number} of $G$.
In the beginning, group theory completely covered this concept, that means, the proof techniques depended on the properties of the automorphism group of graphs, and also the distinguishing number was attributed to the automorphism group of graphs. 
Specialists in group theory developed this concept in such a way that this concept is studied when a group acts on an arbitrary set, instead of focusing only on the distinguishing number when the automorphism group of a graph acts on its vertices. Many works have been done in this way, for example, we can refer to the papers \cite{Babai, Bailey, Chan, Klavžar, Tymoczko}. 
In another research line, experts in graph theory approached this problem. Here, instead of the properties of the automorphism group of a graph, the technique of proofs was changed to the properties of the automorphism of the graph, such as distance preservation. Also, the distinguishing number was assigned to the graph instead of the automorphism group of it. This change of approach also had some advantages, because knowing the automorphism group of a graph is itself a difficult problem, and with this method we can compute the distinguishing number of graphs whose automorphism group is unknown. For some interesting papers using this approach, see \cite{Ahmadi, Collins, Collins2, Tucker,  Kalinowski, Imrich, Klavžar, Russell, Shekarriz}. 
This concept has been the inspiration for other concepts, such as 
 Distinguishing Maps, Proper Distinguishing Coloring, Distinguishing Arc-coloring, Distinguishing Index, and Distinguishing Threshold \cite{Tucker, Collins, Kalinowski2, Kalinowski, Shekarriz}. 

These two important concepts have gone their research paths for years without paying attention to each other. Here we show that there is a connection between them. In fact, we will show that solving the metric dimension problem breaks the symmetry in graphs. This connection makes the resolving set even more important. Since many articles have been written on these two concepts, this connection can yield significant direct  results for both of them. For example, the upper bounds for the metric dimension are also upper bounds for the distinguishing number. Also, the lower bounds for the distinguishing number minus one are lower bounds for the metric dimension. This link is a shortcut for studying them. Although the proof of this connection is not complicated, it can avoid unnecessary complications in their study in some problems. In this paper, we have used the results in the metric dimension to characterize all graphs 
$G$ 
of order 
$n$ 
with the property that 
$D(G) \in \{n-1, n-2\}$. 
We have shown that there are graphs for which these two parameters are freely variable and not dependent. 
Hence we construct graphs $G$ such that $D(G)=n$ and ${\rm dim}(G)=m$ for all values of $n$ and $m$, where $1\leq n< m$.
For any graph $G$, let 
$G_c = G$ 
if 
$G$ 
is connected,
and 
$G_c = \overline{G}$  
if 
$G$ 
is disconnected. 
Let 
$G^{\ast}$ 
denote the twin graph obtained from 
$G$ 
by contracting any maximal set of vertices with the same open or close neighborhood into a vertex. 
The family of graphs of order $n$ and distinguishing number $n-3$ is relatively large and we have obtained 42 of them. 
Let 
{\rsfs F} 
be the set of all graphs except graphs $G$ with the property that 
${\rm dim}(G_c)=|V(G)|-4$, 
${\rm diam}(G_c) \in \{2, 3\}$ 
and 
$5\leq |V(G_{c}^{\ast})| \leq 9$. 
We characterize all graphs
$G \in$ {\rsfs F} 
of order 
$n$
with 
$D(G)= n-3$. 

Since we are studying the structure of the graph, we consider isomorphic graphs to be the same and use `$=$' instead of `$\cong$'.
All definitions and symbols used and undefined are standard and can be found in \cite{book}.
\label{sec:intro}

\section{General Results} 
In this section, we show that in connected graphs, the distinguishing number is bounded by the metric dimension plus one. This bound is sharp and gives a method for distinguishing vertices of graphs by resolving sets. In addition, we show that according to this bound, the metric dimension and the distinguishing number can take arbitrary values. 

\begin{theorem}\label{main}
Let 
$G$ 
be a connected graph. Then 
$D(G)\leq {\rm dim}(G)+1$.
\end{theorem}
\begin{proof}
Let 
$S \subseteq V(G)$ 
be a resolving set of 
$G$. 
Color members of 
$S$ 
with distinct colors 
$1$ 
to 
$|S|$, 
and assign color 
$|S|+1$ 
to the rest of vertices of 
$G$. 
We claim that this coloring is a distinguishing coloring for 
$G$. 
For a contradiction, assume that there exists a non-trivial automorphism 
$f$ 
that preserves this coloring. Thus, all vertices of 
$S$ 
are fixed by 
$f$. 
Since 
$f$ 
is non-trivial, there exist vertices 
$v, u \in V(G)\setminus S$ 
such that 
$f(v)=u$. 
Let 
$w \in S$. 
So, 
$f(w)=w$. 
Since 
$f$
is an isometry, i.e., it preserves distances, we have
$$ 
{\rm d}(v, w)={\rm d}(f(v), f(w))={\rm d}(u, w).
$$ 
This implies that there is no vertex in 
$S$ 
that resolves 
$v$ 
and 
$u$, 
which leads to a contradiction.
\end{proof}
Note that this bound is sharp. For instance, let 
$G=P_n$ 
for an integer 
$n\geq 2$. 
Hence, 
$D(G)={\rm dim}(G)+1=2$.
\begin{corollary}
Let 
$G$ 
be a connected graph. For any resolving set 
$S=\{v_1, v_2, \ldots, v_n \}$ 
of 
$G$, 
$\{\{v_1\}, \{v_2\}, \ldots, \{v_n\}, V(G)\setminus S \}$ 
is a partition of 
$V(G)$ 
into a distinguishing coloring. 
\newpage
\end{corollary} 
As natural bounds, we know that for any connected graph 
$G$ 
of order 
$n$,
$1\leq {\rm dim}(G) \leq n-1$.
This upper bound was improved to $n-{\rm diam}(G)$ by Yushmanov \cite{yush}. So the following result is immediate. 

\begin{corollary}
For any connected graph 
$G$ 
of order 
$n$, 
$D(G)\leq n-{\rm diam}(G)+1$.
\end{corollary} 
For two vertices 
$v$ 
and 
$u$ 
of a graph, 
$v-u$
path is a path between 
$v$ 
and 
$u$.
\begin{theorem}\label{w}
For any $1\leq n< m$, there exists a graph $G$ with $D(G)=n$ and ${\rm dim}(G)=m$.
\end{theorem} 
\begin{proof}
For an integer 
$k\geq 3$, 
let  
$T_k$ 
be a rooted tree with root 
$r$ 
such that 
(1) 
$\deg(r)=k$, 
(2) 
$T_k$ 
has
$k$ 
pendant vertices 
$v_1, v_2, \ldots, v_k$, 
(3) 
${\rm d}(v_1, r)=1, {\rm d}(v_2, r)=2, \ldots, {\rm d}(v_k, r)=k$, 
and 
(4) 
the other non-pendant vertices have degree 
$2$. 
Let  
$n=1$.
For any 
$m\geq 2$, 
let 
$G=T_{m+1}$. 
Clearly, 
$G$ 
is an asymmetric graph and 
$D(G)=1$. 
We claim that 
${\rm dim}(G)=m$. 
Assume that 
$S$ 
is a subset of 
$N(r)$ 
with size 
$m$. 
In what follows, we show that 
$S$ 
is a resolving set of
$G$. 
For this purpose, suppose that 
$x$ 
and 
$y$ 
are two arbitrary vertices of 
$G$. 
Let 
$x$ 
and 
$y$  
be two vertices in an 
$r-v_i$
path for some 
$2\leq i \leq m+1$. 
Since  
$x$ 
and 
$y$ 
have different distances from 
$r$, 
they have different distances from vertices of 
$N(r)$, 
as well. 
Thus, in the two cases whether if there is a vertex of the $r-v_i$ path in 
$S$ 
or not, these vertices are resolved by 
$S$. 
Assume that 
$x$ 
is in an 
$r-v_i$
path and
$y$  
is in an 
$r-v_j$
path, for 
$1\leq i, j \leq m+1$. 
If 
$x$ 
and 
$y$
have different distances from 
$r$, 
same as before 
$x$ 
and 
$y$  
are resolved by some vertices of 
$S$. 
So, suppose that 
$x$ 
and 
$y$ 
have the same distance from 
$r$. 
At least one of 
$r-v_i$
or 
$r-v_j$ 
paths have one vertex, say
$w$, 
in 
$S$. 
One can check that 
$x$ 
and 
$y$  
are resolved by 
$w$. 
This means that 
$S$ 
is a resolving set for 
$G$. 
Now, we will prove that 
${\rm dim}(G)=|S|=m$.
For a contradiction, assume that there is a resolving set
$S'$ 
of 
$G$ 
such that 
$|S'|<m$.
Hence, 
$S'$ 
has nothing in common with the vertex set of at least two 
$r-v_i$
and
$r-v_j$
paths for some 
$1\leq i, j \leq m+1$. 
Let 
$P_i$ 
and 
$P_j$ 
denote the 
$r-v_i$
and
$r-v_j$
paths, respectively. Let 
$N(r)\cap (V(P_i)\cup V(P_j))=\{x, y\}$. 
Thus, the vertices
$x$ 
and 
$y$ 
have the same distance from the vertices of 
$S'$. 
This implies that 
${\rm dim}(G)=m$. 

Assume that 
$2\leq n < m$. 
Let 
$G$ 
denote the graph obtained by joining the root of 
$T_{m-n+2}$ 
to all vertices of a complete graph 
$K_n$. 
Now, one can check that 
$D(G)=n$ 
and 
${\rm dim}(G)=m$. 
\end{proof} 

\section{Graphs $G$ with $D(G)=|V(G)|-1$} 
All connected graphs with ${\rm dim}(G) \in \{|V(G)| - 1, |V(G)|-2\}$  are characterized in \cite{0}. In this section, according to these results, we obtain the graphs  with
$D(G)=|V(G)|-1$.

\begin{remark}\label{r1}
For any graph
$G$,  
$D(G)=|V(G)|$ 
if and only if 
$G \in \{K_{|V(G)|}, \overline{K}_{|V(G)|} \}$.
\end{remark} 
The next corollary is immediate from \cite[Theorem 3]{0} and Remark \ref{r1}.
\begin{corollary}
For any graph 
$G$,  
$D(G)={\rm dim}(G)+1=|V(G)|$ 
if and only if 
$G= K_{|V(G)|}$.
\end{corollary}
\newpage
\begin{lemma}\label{discon}
For positive integers 
$n$ 
and 
$t$, 
let 
{\rsfs A}$= \{G_1, G_2, \ldots , G_n\}$ 
be the set of all connected graphs such that 
$H \in${\rsfs A}
if and only if 
$D(H)=t$. 
Then 
{\rsfs B}$= \{\overline{G_1}, \overline{G_2}, \ldots , \overline{G_n}\} \setminus${\rsfs A}  
is the set of all disconnected graphs such that 
$H \in${\rsfs B} 
if and only if 
$D(H)=t$. 
\end{lemma} 
\begin{proof} 
Let 
{\rsfs C}$=\{\overline{G_1}, \overline{G_2}, \ldots , \overline{G_n}\}$.
Since the distinguishing number of a graph and its complement are the same, it follows that for any 
$H \in${\rsfs C}, 
$D(H)=t$. 
On the other hand, the complement of a disconnected graph must be connected. This means that there is no disconnected graph 
$H$ 
with 
$D(H)=t$, 
except graphs in 
{\rsfs C}. 
However, some members of {\rsfs C} may be connected graphs. These graphs must be in {\rsfs A}. So if we remove them, we will have all disconnected graphs with $D(H)=t$. 
Therefore
{\rsfs B} = {\rsfs C} $\setminus$ {\rsfs A}
is the set of all disconnected graphs such that 
$H \in$ {\rsfs B} 
if and only if 
$D(H)=t$. 
\end{proof} 

\begin{lemma}\label{lemn-2}
Let 
$G$ 
be a connected graph of order 
$n\geq4$ 
and 
${\rm dim}(G) = n - 2$. 
For an integer 
$\ell \geq 1$, $D(G)=n-\ell$ 
if and only if 
$G$ 
is one of the following: 
\begin{itemize}
\begin{small}
\begin{multicols}{2}
\item[(a)] 
$K_{\ell+1, \ell+1}$
\item[(c)] 
$K_{\ell}+\overline{K_t}$, $t\geq \ell$
\item[(e)] 
$K_{\ell-1}+(K_t \cup K_{1})$, $t\geq \max \{2, \ell-1\}$
\item[(b)] 
$K_{t, \ell}$, $t\geq \ell+1$
\item[(d)] 
$K_t + \overline{K_{\ell}}$, $t\geq \ell\geq 2$
\item[(f)]
$K_{t}+(K_{\ell-1}\cup K_{1})$, $t\geq \max \{2, \ell-1\}$
\end{multicols}
\end{small}
\end{itemize}
\end{lemma}
\begin{proof}
By \cite[Theorem 4]{0}, assume first that 
$G= K_{s,t} (s,t\geq 1)$. 
If 
$s=t$, 
then 
$D(G)=t+1$. 
The assumption
$D(G)=n-\ell$ 
concludes that
$t+1=s+t-\ell$. 
So, 
$s=t=\ell+1$.  
This implies that 
$G= K_{\ell+1, \ell+1}$ 
and the result in (a) is obtained. 
If 
$s\neq t$, 
then 
$D(G)=\max \{s, t\}$. 
Let 
$\max \{s, t\}=t$. 
Then 
$t=s+t-\ell$, 
and so 
$s=\ell$. 
Hence, 
$G=K_{t,\ell}$ 
for 
$t\geq \ell+1$. 
Therefore, we have the result in (b). 

Let $G=K_s + \overline{K_t} (s\geq 1, t\geq 2)$. 
One can see that 
$D(G)=\max \{s, t\}$. 
If 
$\max \{s, t\}=t$, 
then 
$t=s+t-\ell$ 
and so 
$s=\ell$. 
This implies that 
$G=K_{\ell}+ \overline{K_t}$,  
and 
$G$ 
is the graph in (c). 
If 
$\max \{s, t\}=s$, 
then 
$t=\ell$ 
and $G$ is the graph in (d).

Finally, suppose that 
$G= K_s + (K_1 \cup K_t) (s,t \geq 1)$. 
Thus, 
$D(G)=\max \{s, t\}$. 
If 
$\max \{s, t\}=t$ ($\max \{s, t\}=s$), 
then 
$s=\ell -1$ ($t=\ell -1$). 
In such a situation, we have obtained the graphs in 
(e) 
and 
(f).
\end{proof}

\begin{theorem}\label{then-2}
Let 
$G$ 
be a graph of order 
$n$. 
Then 
$D(G)=n-1$ 
if and only if 
$G$ 
is one of the following: 
\begin{itemize}
\begin{small}
\begin{multicols}{2}
\item[(1)] 
$C_4$ 
\item[(3)] 
$2K_{2}$
\item[(2)] 
$K_{t, 1}$
\item[(4)] 
$K_{t}\cup K_1$
\end{multicols}
\end{small}
\end{itemize}
where 
$t\geq 2$.
\end{theorem}
\begin{proof} 
Assume first that 
$G$ 
is a connected graph. By Theorem \ref{main}, the metric dimension of 
$G$ 
is equal to 
$n-1$ 
or 
$n-2$. 
\cite[Theorem 3]{0} and Remark \ref{r1} conclude that 
${\rm dim}(G)\neq n-1$. 
So, it suffices to check the graphs 
$G$ 
with 
${\rm dim}(G)= n-2$.
For this, let 
$\ell=1$  
in Lemma \ref{lemn-2}. The graphs in (a) and (b) in Lemma \ref{lemn-2} are the graphs in (1) and (2), respectively. Also, the graph in (c) in Lemma \ref{lemn-2} is appeared in (2). 
Since 
$\ell =1$, $G$ 
will not be equal to any graphs in (d), (e) and (f) in Lemma \ref{lemn-2}.

Now, by Lemma \ref{discon}, we have graphs in (3) and (4) in the characterization as the complements of graphs in (1) and (2), respectively. 
(Note that the only graph with less than $4$ vertices that applies to the assumptions of the theorem is $P_3$, which itself and its complement are in (2) and (4), respectively.)
\end{proof} 
\section{Graphs $G$ with $D(G)=|V(G)|-2$} 
In this section, we investigate graphs 
$G$ 
with 
$D(G)=|V(G)|-2$. 
Our method is based on the use of existing characteristics for the metric dimension and using Theorem \ref{main}. For this purpose, 
we use \cite[Theorem 2.14]{Hernando} and \cite[Theorem 1]{Jannesari} in which the metric dimension of the graphs with the given diameter is specified, 
and the graphs with metric dimension $n-3$ are characterized, respectively. 
In the following, we will express some definitions of the required prerequisites.
Let 
$G$ 
be a graph and 
$v, u \in V(G)$. 
The vertices 
$v$ 
and 
$u$ 
are {\it twin} if 
${\rm N}(v) \setminus \{u\}={\rm N}(u) \setminus \{u\}$. 
Let 
$u \equiv v$ 
if and only if 
$u = v$ 
or 
$u, v$ 
are twins. 
Let 
$v^{\ast}= \{u \in V(G) \; | \; u \equiv v\}$. 
The graph $G^{\ast}$, known as the twin graph of $G$,  
is the graph with 
$V(G^{\ast})=\{v^{\ast} \; | \; v \in V(G) \}$ 
and 
$E(G^{\ast})=\{v^{\ast} u^{\ast} \; | \; uv \in E(G) \}$.  
In fact, 
$u$ 
and 
$v$ 
are adjacent in 
$G$ 
if and only if each vertex in 
$u^{\ast}$ is adjacent to each vertex in
$v^{\ast}$. 
The graph 
$G[v^{\ast}]$ 
can only have two types of structures, which are divided into three types here. 

\begin{center}
$G[v^{\ast}] \cong \begin{cases}
K_1 & \text{if}\; v^{\ast}\; \text{is of type (1),}\; \\
K_r, r\geq 2 & \text{if}\; v^{\ast}\; \text{is of type (K),}\; \\ 
\overline{K_r}, r\geq 2 & \text{if}\; v^{\ast}\; \text{is of type (N).} 
\end{cases}$ 
\end{center}
Also, we say that 
$v^{\ast} \in G^{\ast}$
is of type (1K) if 
$v^{\ast}$
is of type (1) or (K), of type (1N) if 
$v^{\ast}$
is of type (1) or (N), of type (KN) if 
$v^{\ast}$
is of type (K) or (N), and of type (1KN) if 
$v^{\ast}$
is of type (1), (K) or (N).
Let 
$\alpha (G^{\ast})$
denote the number of vertices
of 
$G^{\ast}$
of type (K) or (N). 
It is clear that the graph 
$G$ 
is uniquely defined by the graph 
$G^{\ast}$. 
So we can express our classification on the graph 
$G^{\ast}$. For more details about twin graphs, you can refer to \cite{Hernando, Jannesari}. 

Let $(v_1, v_2, \ldots, v_n)$ denoted the path $P_n$ on the vertices $v_1, v_2, \ldots, v_n$. 

\begin{observation} \label{obser}
For positive integers 
$n$ 
and 
$m$, 
let 
$G$ 
be a connected graph of order 
$n$.  
Let 
$V(G^{\ast})=\{v^{\ast}_1, v^{\ast}_2, \ldots, v^{\ast}_{|V(G^{\ast})|}\}$  
and 
$D(G) =|v^{\ast}_i|$  
for a 
$i (1\leq i \leq |V(G^{\ast})|)$.
If 
$|v^{\ast}_1|+|v^{\ast}_2|+ \ldots +|v^{\ast}_{i-1}|+|v^{\ast}_{i+1}|+\ldots +|v^{\ast}_{|V(G^{\ast})|}| \neq m$,  
then 
$D(G) \neq n-m$.
\end{observation}
For any graph 
$G$, 
we say 
$G^{\ast}$ 
is an {\it almost asymmetric} graph if there is no non-trivial automorphism of 
$G$ 
that images some vertices of 
$v^{\ast}_i$ 
to 
some vertices of 
$v^{\ast}_j$  
for any two vertices 
$v^{\ast}_i$ 
and
$v^{\ast}_j$ 
of 
$G^{\ast}$. 
Clearly, if 
$G^{\ast}$ 
is almost asymmetric, then 
$D(G)= \max\limits_{v^{\ast} \in V(G^{\ast})} \{|v^{\ast}| \; | \; v^{\ast} \in G^{\ast} \}$. 
\begin{corollary}\label{lem-almost-asymmetric}
For positive integers 
$n$ 
and 
$m$, 
let 
$G$ 
be a connected graph of order 
$n$. 
If 
$G^{\ast}$ 
is almost asymmetric and 
$n - \max\limits_{v^{\ast} \in V(G^{\ast})} \{|v^{\ast}| \; | \; v^{\ast} \in G^{\ast} \}  \neq m$, 
then 
$D(G) \neq n-m$.
\end{corollary} 
\newpage
In the following, in  Lemmas \ref{G1}, \ref{G2}, \ref{G3}, \ref{G5}, \ref{G6}, \ref{Gal}, we investigate the distinguishing number of the graphs introduced in \cite[Theorem 1]{Jannesari}. 
\begin{lemma}\label{G1}
If $G^{\ast}$ is graph $G_1$ in \cite[Theorem 1]{Jannesari}, then
\begin{itemize} 
\item[(1)] 
$D(G)=|V(G)|-2$ 
if and only if 
$G$ 
is one of the following: 
\begin{itemize}
\begin{small}
\begin{multicols}{2}
\item[(1a)] 
$K_{1, 1, t}$, $t\geq 2$
\item[(1b)] 
$K_{1, 2, 2}$
\end{multicols}
\end{small}
\end{itemize} 
\item[(2)] 
$D(G)=|V(G)|-3$ 
if and only if 
$G$ 
is one of the following: 

\begin{itemize}
\begin{small}
\begin{multicols}{2}
\item[(2a)] 
$K_{1, 2, t}$, $t\geq 3$
\item[(2c)] 
$K_2 + K_{2, 2}$ 
\item[(2b)] 
$K_{1, 3, 3}$
\item[(2d)] 
$K_{2, 2, 2}$
\end{multicols}
\end{small}
\end{itemize} 
\end{itemize}
\end{lemma}
\begin{proof} 
Let 
$\ell \in \{2,3\}$.
Assume first that  
$v^{\ast}_1$ 
is a vertex of $G^{\ast}$ of type (1). Then 
$v^{\ast}_2$ 
and 
$v^{\ast}_3$ 
are vertices of $G^{\ast}$ of type (N). If 
$|v^{\ast}_2|\neq |v^{\ast}_3|$, 
then   
$D(G)=\max \{|v^{\ast}_2|, |v^{\ast}_3|\}$. 
Let 
$\max \{|v^{\ast}_2|, |v^{\ast}_3|\}=|v^{\ast}_2|$. 
Hence, 
$D(G)=|V(G)|-\ell$ 
concludes that  
$|v^{\ast}_3|=\ell-1$. 
Then we have the graphs in (1a) and (2a) when $\ell = 2$ and $\ell = 3$, respectively. 
If 
$|v^{\ast}_2|=|v^{\ast}_3|$, 
$D(G)=|v^{\ast}_2|+1$. 
So, the equation 
$|v^{\ast}_2|+1=|v^{\ast}_2|+|v^{\ast}_3|+1-\ell$ 
results in 
$|v^{\ast}_3|=\ell$. 
In this situation, we have obtained the  graphs in (1b) and (2b) when $\ell = 2$ and $\ell = 3$, respectively. 

Let 
$v^{\ast}_1$ 
be of type (K). 
If 
$|v^{\ast}_2|\neq |v^{\ast}_3|$, 
then   
$D(G)=\max \{|v^{\ast}_1|, |v^{\ast}_2|, |v^{\ast}_3|\}$. 
Without loss of generality, we may assume that 
$\max \{|v^{\ast}_1|, |v^{\ast}_2|, |v^{\ast}_3|\}=|v^{\ast}_1|$. 
The assumption
$D(G)=|V(G)|-\ell$ 
concludes that 
$|v^{\ast}_2|+|v^{\ast}_3|=\ell$, 
which is impossible. 
If 
$|v^{\ast}_2|= |v^{\ast}_3|$, 
then   
$D(G)=\max \{|v^{\ast}_1|, |v^{\ast}_2|+1\}$. 
If  
$D(G)=|v^{\ast}_1|$,   
then the size of $|v^{\ast}_2|$ will be equal to 1 or $3/2$, which is not possible. 
If 
$D(G)=|v^{\ast}_2|+1$, 
then 
$|v^{\ast}_1| + |v^{\ast}_3|=\ell +1$. 
This implies that 
$\ell = 3$ 
and 
$|v^{\ast}_1|= |v^{\ast}_2|= |v^{\ast}_3|=2$. 
This is the graph in (2c).

Suppose that 
$v^{\ast}_1$ 
is of type (N). 
If 
$|v^{\ast}_1|\neq |v^{\ast}_2|\neq |v^{\ast}_3| \neq |v^{\ast}_1|$, 
then   
$D(G)=\max \{|v^{\ast}_1|, |v^{\ast}_2|, |v^{\ast}_3|\}$. 
Let
$\max \{|v^{\ast}_1|, |v^{\ast}_2|, |v^{\ast}_3|\}=|v^{\ast}_1|$. 
Hence,
$|v^{\ast}_2|+|v^{\ast}_3|=\ell$, 
which is impossible.  
If 
$|v^{\ast}_1|=|v^{\ast}_2|\neq |v^{\ast}_3|$,  
then
$D(G)=\max \{|v^{\ast}_1|+1, |v^{\ast}_3|\}$. 
If 
$D(G)=|v^{\ast}_3|$, 
then 
$|v^{\ast}_1| \in \{1, 3/2\}$, 
which contradicts the definition of type (N).
If 
$D(G)=|v^{\ast}_1|+1$, 
then 
$\ell =3$ 
and so 
$|v^{\ast}_2|= |v^{\ast}_3|=2$, 
a contradiction.
If 
$|v^{\ast}_1|=|v^{\ast}_2|= |v^{\ast}_3|$,  
then
$D(G)=|v^{\ast}_1|+1$ 
and we have obtained the graph in (2d).
\end{proof}
\begin{lemma}\label{G2}
If $G^{\ast}$ is graph $G_2$ in \cite[Theorem 1]{Jannesari}, then
$D(G)\neq |V(G)|-2$; 
and 
$D(G)= |V(G)|-3$
if and only if 
$G$ 
is one of the following: 
\begin{itemize}
\begin{small}
\begin{multicols}{2}
\item[(a)] 
$\overline{K_2}+(K_1 \cup K_t )$, $t\geq 2$
\item[(c)] 
$\overline{K_2} + 2K_{2}$
\item[(b)] 
$\overline{K_t}+(K_1 \cup K_2)$, $t\geq 2$
\item[(d)] 
$K_2 + 2K_{2}$
\end{multicols}
\end{small}
\end{itemize}
\end{lemma}
\begin{proof} 
Let 
$\ell \in \{2,3\}$ 
and 
$D(G)= |V(G)|-\ell$.
In $G_2$(a), let 
$v^{\ast}_1$, 
$v^{\ast}_2$ 
and
$v^{\ast}_3$
be the vertices of types 
(K), (N) 
and any type, respectively. Suppose first that 
$v^{\ast}_3$ 
is of type (1). Thus 
$D(G)=\max \{|v^{\ast}_1|, |v^{\ast}_2|\}$. 
If  
$\max \{|v^{\ast}_1|, |v^{\ast}_2|\}= |v^{\ast}_1|$, 
then 
$|v^{\ast}_1|= |v^{\ast}_1|+|v^{\ast}_2|+1-\ell$ 
and so 
$|v^{\ast}_2|=\ell - 1$. 
This implies that 
$\ell = 3$ 
and 
$G= \overline{K_2}+(K_1 \cup K_t )$, $t\geq 2$. 
Therefore we will have the graph in (a). Also, if 
$\max \{|v^{\ast}_1|, |v^{\ast}_2|\}= |v^{\ast}_2|$, 
then 
$G= \overline{K_t}+(K_1 \cup K_2)$ with $t\geq 2$, 
the graph in (b). Let 
$v^{\ast}_3$ 
be of type (N). 
Then 
$D(G)=\max \{|v^{\ast}_1|, |v^{\ast}_2|, |v^{\ast}_3|\}$. 
Without loss of generality, we may assume that 
$\max \{|v^{\ast}_1|, |v^{\ast}_2|, |v^{\ast}_3|\}=|v^{\ast}_1|$. 
So, 
$|v^{\ast}_2|+|v^{\ast}_3|=\ell$, 
which is a contradiction. 
Assume that 
$v^{\ast}_3$ 
is of type (K). 
Clearly, 
$|v^{\ast}_1|=|v^{\ast}_3|$. 
Hence, 
$D(G)=\max \{|v^{\ast}_1|+1, |v^{\ast}_2|\}$. 
If 
$\max \{|v^{\ast}_1|+1, |v^{\ast}_2|\}=|v^{\ast}_2|$, 
then it is impossible as before. If 
$\max \{|v^{\ast}_1|+1, |v^{\ast}_2|\}=|v^{\ast}_1|+1$, 
then 
$\ell=3$ 
and 
$|v^{\ast}_1|= |v^{\ast}_2|= |v^{\ast}_3|=2$. 
Then, $G$ is the graph in (c).

In $G_2$(b), 
let 
$v^{\ast}_1$, 
$v^{\ast}_2$ 
and
$v^{\ast}_3$
be the vertices of types 
(K), any type
and (NK), respectively. 
If 
$v^{\ast}_3$ 
is of type (N), then the size of some vertices of $G^{\ast}$ will be equal to one, which is impossible. 
Let 
$v^{\ast}_3$ 
be of type (K). If 
$|v^{\ast}_1|\neq |v^{\ast}_3|$, 
then we reach the same contradiction as before. Thus 
$|v^{\ast}_1| = |v^{\ast}_3|$ 
and so 
$D(G)=\max \{|v^{\ast}_1|+1, |v^{\ast}_2|\}$. 
One can check that 
$D(G)\neq|v^{\ast}_2|$. 
So, 
$D(G)=|v^{\ast}_1|+1$. 
This concludes that 
$\ell=3$ 
and 
$|v^{\ast}_1|= |v^{\ast}_2|= |v^{\ast}_3|=2$. 
Therefore, 
$G=K_2 + 2K_{2}$ 
and we found the graph in (d). (Note that when $v^{\ast}_2$ is of type (N), we have already considered it in (c).)
\end{proof}
\begin{lemma}\label{G3}
If $G^{\ast}$ is graph $G_3$ in \cite[Theorem 1]{Jannesari}, then
$D(G)\neq |V(G)|-2$; 
and 
$D(G)= |V(G)|-3$
if and only if 
$G=K_1+(K_1 \cup (K_{1,t}))$ with $t\geq 2$. 
\end{lemma}
\begin{proof} 
Let 
$v^{\ast}_1, v^{\ast}_2, v^{\ast}_3$ 
and 
$v^{\ast}_4$ 
be the degree-$2$ vertex of type (1K), the degree-$2$ vertex of type (N), the degree-$3$ vertex and the leaf, respectively.
Let 
$\ell \in \{2,3\}$ 
and 
$D(G)= |V(G)|-\ell$.
In all cases, there is no automorphism of $G$ that images some vertices of 
$v^{\ast}_1$ 
to some vertices of 
$v^{\ast}_2$. 
Hence, 
$D(G)=\max \{|v^{\ast}_1|, |v^{\ast}_2|, |v^{\ast}_3|, |v^{\ast}_4|\}$. 
If 
$D(G)=|v^{\ast}_i|$, $i\in \{1, 3, 4 \}$, 
then 
$D(G)=n-\ell$ 
concludes that some of 
$|v^{\ast}_i|$'s 
must be equal to one or zero. 
It is impossible. 
If 
$D(G)=|v^{\ast}_2|$, 
then 
$\ell = 3$, 
$|v^{\ast}_1|=|v^{\ast}_3|=|v^{\ast}_4|=1$ 
and 
$G=K_1+(K_1 \cup K_{1,t})$ with $t\geq 2$. 
\end{proof}
\begin{lemma}\label{G5}
If $G^{\ast}$ is graph $G_5$ in \cite[Theorem 1]{Jannesari}, then $D(G)\neq |V(G)|-2$; 
and 
$D(G)= |V(G)|-3$
if and only if 
$G=C_5^{'}$.
\end{lemma}
\begin{proof}
Let 
$\ell \in \{2,3\}$ 
and 
$D(G)= |V(G)|-\ell$.
Let 
$v^{\ast}_1$ 
and 
$v^{\ast}_2$ 
be the degree-$3$ vertices, 
$v^{\ast}_3$ 
be the degree-$2$ vertex that is adjacent to 
$v^{\ast}_1$ 
and 
$v^{\ast}_2$, 
and 
$v^{\ast}_4$ 
and 
$v^{\ast}_5$ 
be the adjacent degree-$2$ vertices. Let 
$|v^{\ast}_1|=|v^{\ast}_2|=1$. 
If 
$|v^{\ast}_3|=1$, 
then 
$D(G)= 2$. 
This implies that 
$\ell = 3$ 
and 
$G=C_5^{'}$. 
If 
$|v^{\ast}_3|\neq 1$, 
then 
$D(G)= |v^{\ast}_3|=|V(G)|-4$, 
a contradiction. If 
$|v^{\ast}_1|=|v^{\ast}_2|\neq 1$ 
or 
$|v^{\ast}_1|\neq |v^{\ast}_2|$, 
then 
$D(G)=\max \{|v^{\ast}_1|, |v^{\ast}_2|, |v^{\ast}_3|\}$. 
Therefore, in all cases 
$D(G)\notin \{|V(G)|-2, |V(G)|-3\}$.
\end{proof}
\begin{lemma}\label{G6}
If $G^{\ast}$ is graph $G_6$ in \cite[Theorem 1]{Jannesari}, then
$D(G)\neq |V(G)|-2$; 
and 
$D(G)= |V(G)|-3$
if and only if 
$G=K_1 + P_4$.  
\end{lemma}
\begin{proof}
Let 
$v^{\ast}_1$ 
and 
$v^{\ast}_4$ 
be the degree-$2$ vertices, 
$v^{\ast}_3$ 
and 
$v^{\ast}_5$ 
be the degree-$3$ vertices, and 
$v^{\ast}_2$ 
be the degree-$4$ vertex. 
Let 
$\ell \in \{2,3\}$ 
and 
$D(G)= |V(G)|-\ell$. 
The only non-trivial automorphism of $G^{\ast}$ is the automorphism that images  
$v^{\ast}_1$ 
to
$v^{\ast}_4$, 
$v^{\ast}_4$ 
to
$v^{\ast}_1$, 
$v^{\ast}_3$ 
to
$v^{\ast}_5$, 
$v^{\ast}_5$ 
to
$v^{\ast}_3$, 
and does not move 
$v^{\ast}_2$.
Since the non-adjacent vertices are not of type (K), if 
$|v^{\ast}_1|=|v^{\ast}_4|$, 
then 
$|v^{\ast}_1|=1$. 
Now, if 
$|v^{\ast}_3|=|v^{\ast}_5|=1$, 
then 
$D(G)=\max \{2, |v^{\ast}_2|\}$. 
If 
$D(G)=2$, 
then 
$\ell=3$, 
$|v^{\ast}_2|=1$ 
and so 
$G=K_1 + P_4$.
If
$D(G)=|v^{\ast}_2|$, 
then 
$|v^{\ast}_2|+4-\ell=|v^{\ast}_2|$, 
a contradiction. 
If 
$|v^{\ast}_3|=|v^{\ast}_5|\neq 1$, 
then 
$D(G)=\max \{|v^{\ast}_3|, |v^{\ast}_2|\}$. 
In both cases 
$D(G)= |v^{\ast}_3|$ 
and 
$D(G)=|v^{\ast}_2|$, 
$D(G)\notin \{|V(G)|-2, |V(G)|-3\}$. 
If 
$|v^{\ast}_3|\neq |v^{\ast}_5|$ 
or 
$|v^{\ast}_1| \neq |v^{\ast}_4|$, 
then 
$G$ 
is an almost asymmetric graph. Therefore, Corollary \ref{lem-almost-asymmetric} concludes that 
$D(G)\neq |V(G)|-\ell$.
\end{proof}
\begin{lemma} \label{Gal}
If $G^{\ast}$ is one of the graphs $G_7$, $G_8$, $G_9$ or $G_{10}$ in \cite[Theorem 1]{Jannesari}, 
then $D(G)\notin \{|V(G)|-2, |V(G)|-3\}$.
\end{lemma}
\newpage
\begin{proof}
In (a), $G^{\ast}$ is an almost asymmetric graph. In Corollary \ref{lem-almost-asymmetric}, let 
$m \in \{2, 3\}$.
Hence, $D(G) \neq |V(G)|-m$. 
In (b) and (c), $D(G)= \max\limits_{v^{\ast} \in V(G^{\ast})} \{|v^{\ast}| \; | \; v^{\ast} \in G^{\ast} \}$. 
Since there are no three vertices of type (1), then $D(G)\notin \{|V(G)|-2, |V(G)|-3\}$. 
In (d), if the two degree-$3$ adjacent vertices of type (K) have different sizes, then $G^{\ast}$ is an almost asymmetric graph and 
the result is immediate by Corollary \ref{lem-almost-asymmetric}. 
If the two degree-$3$ adjacent vertices of type (K) have the same size, then 
$D(G) \in \{|V(G)|-2, |V(G)|-3\}$ 
concludes that some vertices of 
$G^{\ast}$ 
have the size zero. It is impossible and the proof is completed.
\end{proof}

Since graphs are not presented in twin form in the final characterization, we will provide a definition to facilitate their display. To do so, let us recall the definition of the blow-up of a graph. Let $G$ be a  graph of order $n$  with vertices $v_1, v_2, \ldots, v_n$, and assume that  $H_1, H_2, \ldots, H_n$ are complete   or empty  graphs. The {\em blow-up} of $G$, denoted  by $G[H_1, H_2,  \ldots, H_n]$, is the graph obtained  as follows:
\begin{itemize}
\item 
Every vertex $v_i$ of $G$ is replaced by $H_i$ for every $i$ with $1\leq i\leq n$.
\item 
For any two vertices $v_i$ and $v_j$ if $v_iv_j\in E(G)$, then for every $u\in V(H_i)$ and every $v\in V(H_j)$, $uv$ is an edge of $G[H_1, H_2, \ldots, H_n]$.
\end{itemize}
Note that for any two vertices $v_i$ and $v_j$ with  $v_iv_j\in E(G)$, if $H_i$ and $H_j$ are empty graphs, then the subgraph induced by $H_i \cup H_j$ in $G[H_1, \ldots, H_n]$ is the complete bipartite graph with parts $H_i$ and $H_j$. 
Specifically, in the path $(v_1, v_2, \ldots, v_n)$, $P_n [H_1, \ldots, H_n]=(H_1, H_2, \ldots, H_n)$. 

In the next Lemma all graphs $G$ of order $n$ with ${\rm diam(G)}\in \{3, 4\}$, ${\rm dim(G)}=n-{\rm diam(G)}$ and $D(G) \in \{n-2,n-3\}$, according to \cite[Theorem 2.14]{Hernando}, are characterized.
\begin{lemma}\label{Hernando-Lem}
Let 
$G$ 
be a connected graph of order 
$n$ 
and diameter 
$d \in \{3, 4\}$. 
If 
${\rm dim(G)}=n-d$,
then $D(G)=n-2$ if and only if $G=P_4$; and $D(G)=n-3$ if and only if $G$ is one of the following: 
\begin{itemize}
\begin{small}
\begin{multicols}{2}
\item[(i)] 
$P_5$
\item[(iii)] 
$P_4 [K_1, K_t, K_1, K_1]$, $t\geq 2$ 
\item[(v)] 
$P_4 [K_1, \overline{K_t}, K_1, K_1]$, $t\geq 2$
\item[(ii)] 
$P_4 [K_t, K_1, K_1, K_1]$, $t\geq 2$
\item[(iv)] 
$P_4 [\overline{K_t}, K_1, K_1, K_1]$, $t\geq 2$
\end{multicols}
\end{small}
\end{itemize}
\end{lemma}
\begin{proof} 
The graph 
$G^{\ast}$ 
is one of the graphs described in \cite[Theorem 2.14]{Hernando}. 
In 1(a), assume first that 
$d=3$.
If 
$\alpha(G^{\ast})=0$,  
then 
$G=P_4$ 
and
$D(G)=n-2$. 
In what follows, we will see that $G=P_4$ is the only graph with $D(G)=n-2$ satisfying the assumptions of this lemma.
If 
$\alpha(G^{\ast})=1$,  
then there are two non-isomorphic positions for placing a vertex of type (K) or (N) in $P_4$, 
which in all cases, the distinguishing number is equal to $n-3$. 
So, we have obtained graphs in (ii), (iii), (iv) and (v). Let 
$d=4$. 
If 
$\alpha(G^{\ast})=0$,  
then 
$G=P_5$ 
and
$D(G)=n-3$. 
It is the graph in (i). 
If 
$\alpha(G^{\ast})=1$, 
then it is easy to check that in all cases 
$D(G)=n-4$. 
In the rest of the cases 1(b), 1(c), 1(d), (2) and (3), there are no graphs with distinguishing numbers $n-2$ or $n-3$, and its proof is straightforward using 
Corollary \ref{lem-almost-asymmetric} (for the graph $G^{\ast}$ with $|V(G^{\ast})|\geq 5$) or techniques similar to the previous lemmas 
(for the graph $G^{\ast}$ with $|V(G^{\ast})|=4$).
\end{proof}
\newpage
\begin{theorem}
Let 
$G$ 
be a graph of order 
$n\geq 4$. 
Then 
$D(G)=n-2$ 
if and only if 
$G$ 
is one of the following: 
\begin{itemize}
\begin{small}
\begin{multicols}{2}
\item[(1)] 
$C_5$ 
\item[(3)] 
$K_{1,2,2}$  
\item[(5)] 
$K_{3,3}$
\item[(7)] 
$K_{t,2}$, $t\geq 3$ 
\item[(9)] 
$K_{2}+ \overline{K_t}, t\geq 2$  
\item[(11)] 
$K_t + \overline{K_2}$, $t\geq 2$ 
\item[(13)] 
$K_1 + (K_t \cup K_1)$, $t\geq 2$
\item[(2)] 
$P_4$
\item[(4)] 
$2K_2 \cup K_1$ 
\item[(6)] 
$2K_3$ 
\item[(8)] 
$K_{t}\cup K_2$, $t\geq 3$ 
\item[(10)] 
$K_t \cup 2K_1$, $t\geq 2$ 
\item[(12)] 
$\overline{K_t}\cup K_2$, $t\geq 2$  
\item[(14)] 
$K_{t, 1}\cup K_1$, $t\geq 2$ 
\end{multicols} 
\end{small}
\end{itemize}
\end{theorem}
\begin{proof}
Let 
$G$ 
be a connected graph with $D(G)=n-2$. Theorem \ref{main}, \cite[Theorem 3]{0} and Remark \ref{r1} imply that it suffices to check the graphs with metric dimension 
$n-2$ 
and 
$n-3$. 
So, assume first that 
${\rm dim}(G)=n-2$. 
According to Lemma \ref{lemn-2}, let
$\ell=2$. 
The graphs in (a), (b), (c), (d) and (e) in Lemma \ref{lemn-2} are the graphs in (5), (7), (9), (11) and (13), respectively.
Note that the graph in (f) in Lemma \ref{lemn-2} is appeared in (11).
Now, let 
${\rm dim}(G)=n-3$. 
Since 
${\rm dim}(G) \leq n-{\rm diam}(G)$, 
${\rm diam}(G)\leq 3$. 
If
${\rm diam}(G)=2$, 
then the graphs in (1), (3) and (9) follow directly from \cite[Theorem 1]{Jannesari} and Lemmas \ref{G1}, \ref{G2}, \ref{G3}, \ref{G5}, \ref{G6}, \ref{Gal}.
If
${\rm diam}(G)=3$, 
then \cite[Theorem 2.14]{Hernando} and Lemma \ref{Hernando-Lem} show that $G$ is the graph in (2). 
If
$G$ 
is a disconnected graph, then by Lemma \ref{discon}, $G$ has the form described in (4), (6), $\ldots ,$ (14) as the complements of graphs in (3), (5), $\ldots ,$ (13), respectively. 
(Note that $C_5$ and $P_4$ are the self-complementary graphs).
\end{proof}

\section{Graphs $G$ with $D(G)=|V(G)|-3$} 

In \cite{yush}, Yushmanov proved that ${\rm dim}(G) \leq n- {\rm diam}(G)$ for any connected graph $G$ of order $n\geq 2$. 
So, if ${\rm dim}(G)=n-4$, then ${\rm diam}(G) \in \{2, 3, 4\}$. 
Graphs with 
${\rm dim}(G)=n-4$ and ${\rm diam}(G)=4$ are characterised in \cite[Theorem 2.14]{Hernando}. 
Recently, a study was done in graphs with the property that  ${\rm dim}(G)=n-4$ and ${\rm diam}(G) \in \{2, 3\}$ in \cite{dim n-4}. 
For the graphs with ${\rm dim}(G)=n-4$, they have proven that 
$4\leq |V(G_{c}^{\ast})| \leq 9$ when 
${\rm diam}(G)=2$, 
and 
$4\leq |V(G_{c}^{\ast})| \leq 7$ when 
${\rm diam}(G)=3$. 
Moreover, they characterised all graphs with ${\rm dim}(G)=n-4$, ${\rm diam}(G) \in \{2, 3\}$ and $|V(G_{c}^{\ast})|=4$. 
Focusing on these results we will obtain a family of graphs $G$ with $D(G)=n-3$.

In the next two lemmas, we will check the graphs in \cite[Theorem 4.1]{dim n-4} whose distinguishing number equals $n-3$.
\begin{lemma}\label{dim n-4 lem1}
If 
$G^{\ast}$ 
is one of graphs (g.1), (g.2), (g.3), (g.4) or (g.5) in \cite[Theorem 4.1]{dim n-4},
then $D(G)=|V(G)|-3$ if and only if $G$ is one of the following: 
\begin{itemize}
\begin{small}
\begin{multicols}{2}
\item[(i)] 
$K_1 + (K_1 \cup 2K_{2})$
\item[(ii)] 
$K_1 + (K_1 \cup K_{2, 2})$
\end{multicols}
\end{small}
\end{itemize}
\end{lemma} 
\begin{proof} 
In (g.1), let 
$v^{\ast}_1$ 
and 
$v^{\ast}_2$ 
be the pendant vertices of type (K), 
$v^{\ast}_3$
be the pendant vertex of type (1NK) and 
$v^{\ast}_4$ 
be the degree-$3$ vertex of 
$G^{\ast}$.
If 
$v^{\ast}_3$ 
is of type (K) and 
$|v^{\ast}_1|=|v^{\ast}_2|=|v^{\ast}_3|$, 
then 
$D(G)= \max \{|v^{\ast}_1|+1, |v^{\ast}_4|\}$. 
Thus 
$D(G)=|V(G)|-\ell$ 
implies that 
$\ell \neq 3$. 
If 
$v^{\ast}_3$ 
is of type (1N) and 
$|v^{\ast}_1|=|v^{\ast}_2|$, 
then 
$D(G)= \max \{|v^{\ast}_1|+1, |v^{\ast}_3|, |v^{\ast}_4|\}$. 
If 
$D(G)= |v^{\ast}_1|+1$, 
then 
$D(G)=|v^{\ast}_1|+ |v^{\ast}_2|+|v^{\ast}_3|+|v^{\ast}_4|- 3$ 
concludes that 
$|v^{\ast}_1|=|v^{\ast}_2|=2$ 
and 
$|v^{\ast}_3|=|v^{\ast}_4|=1$. 
Therefore the result in (i) is obtained. 
By Observation \ref{obser}, 
$D(G)\neq |V(G)|-3$ 
in all other cases in (g.1). 
Also, in (g.2), (g.3) and (g.4), Corollary \ref{lem-almost-asymmetric} concludes that 
$D(G)\neq |V(G)|-3$. 
In (g.5), let 
$v^{\ast}_1$ 
and 
$v^{\ast}_2$ 
be the degree-$3$ vertices of type (N), 
$v^{\ast}_3$
be the degree-$3$ vertex and 
$v^{\ast}_4$ 
be the pendant vertex of 
$G^{\ast}$. 
If 
$|v^{\ast}_1|=|v^{\ast}_2|$, 
then 
$D(G)= \max \{|v^{\ast}_1|+1, |v^{\ast}_3|, |v^{\ast}_4|\}$. 
If 
$D(G)= |v^{\ast}_1|+1$, 
then 
$D(G)= |V(G)|-3$ 
if and only if 
$|v^{\ast}_2|=2$,
$|v^{\ast}_3|=|v^{\ast}_4|=1$ 
and so 
$G=K_1 + (K_1 \cup K_{2, 2})$. 
This graph is appeared in (ii). 
In all other cases, one can check that Corollary \ref{lem-almost-asymmetric} concludes the results.
\end{proof}
\begin{lemma}\label{dim n-4 lem2}
If 
$G^{\ast}$ 
is one of graphs (g.6), (g.7), (g.8) or (g.9) in \cite[Theorem 4.1]{dim n-4},
then 
$D(G)\neq |V(G)|-3$.
\end{lemma} 
\begin{proof}
In (g.6), 
let 
$v^{\ast}_1$ 
and 
$v^{\ast}_2$
be the vertices of type (K),
$v^{\ast}_3$ 
be the vertex of type (1NK) and
$v^{\ast}_4$ 
be the vertex of type (NK). 
If 
$v^{\ast}_3$ 
and
$v^{\ast}_4$ 
are of type (K) and 
$|v^{\ast}_1|=|v^{\ast}_2|=|v^{\ast}_3|=|v^{\ast}_4|$, 
then we claim that 
$D(G)=|v^{\ast}_1|+1$.
Since 
$\binom{|v^{\ast}_1|+1}{|v^{\ast}_1|}\geq 3$,  
there are at least three subsets $A$, $B$ and $C$ of
$\{1, 2, \ldots, |v^{\ast}_1|+1 \}$ 
of size 
$|v^{\ast}_1|$ 
that can be used as member colors of each vertex of 
$G^{\ast}$. 
Now, color two adjacent vertices of $G^{\ast}$ using the colors from $A$, ensuring that every member of each vertex of $G^{\ast}$ is assigned a unique color.
In the same way, color members of each of the other two vertices of $G^{\ast}$ with the colors in $B$ and $C$, respectively. 
Hence, one can see that this coloring is a distinguishing coloring of $G$ and so $D(G)=|v^{\ast}_1|+1$. 
This implies that 
$D(G)\neq |V(G)|-3$. 
In all other cases, one can check that 
$D(G)\in \{|v^{\ast}_i|, |v^{\ast}_i|+1\}$  
for some $1 \leq i \leq 4$,
and so 
$D(G)\neq |V(G)|-3$. 
In (g.7) and (g.8), 
it is easy to see that 
$|v^{\ast}_i| \leq D(G) \leq |v^{\ast}_i| + 1$ 
for a $v^{\ast}_i \in G^{\ast}$. 
However, in all cases, the sum of the sizes of all three vertices of $G^{\ast}$ is greater than $4$. 
This means that 
$D(G)\neq |V(G)|-3$.
In (g.9), let 
$v^{\ast}_1$ 
be the vertex of any type and 
$v^{\ast}_2, v^{\ast}_3$ 
and 
$v^{\ast}_4$ 
be the other vertices. 
If 
$v^{\ast}_1$ 
is of type (N) and 
$|v^{\ast}_1|=|v^{\ast}_2|=|v^{\ast}_3|=|v^{\ast}_4|\geq 3$, 
then 
$\binom{|v^{\ast}_1|+1}{|v^{\ast}_1|}\geq 4$ 
and so
$D(G)= |v^{\ast}_1|+1$. 
This implies that 
$D(G)\neq |V(G)|-3$. 
If 
$|v^{\ast}_1|=|v^{\ast}_2|=|v^{\ast}_3|=|v^{\ast}_4|= 2$, 
then 
$D(G)= |v^{\ast}_1|+2$.  
Thus 
$D(G) =|V(G)|-3$ 
if and only if 
$|v^{\ast}_1|=5/3$, 
a contradiction. In all other cases, $D(G)\neq |V(G)|-3$ and the argument is similar to the previous one or follows from Observation \ref{obser}.
\end{proof} 
\begin{lemma}\label{dim n-4 lem3}
Let 
$G$ 
be a connected graph of order 
$n$. 
If 
${\rm dim}(G)=n-4$, 
${\rm diam}(G)=3$ 
and 
$|V(G^{\ast})|=4$, 
then 
$D(G) \neq n-3$.
\end{lemma}
\begin{proof} 
The graph 
$G^{\ast}$ 
is one of the graphs described in \cite[Theorem 3.1]{dim n-4}.
Let 
$G^{\ast}$ 
be 
$(v^{\ast}_1, v^{\ast}_2, v^{\ast}_3, v^{\ast}_4)$. 
In (g.1), if 
$v^{\ast}_1$ (resp. $v^{\ast}_2$)
and 
$v^{\ast}_4$ (resp. $v^{\ast}_3$)
are of the same type,  
and 
$|v^{\ast}_1|=|v^{\ast}_2|=|v^{\ast}_3|=|v^{\ast}_4|$, 
then 
$D(G)=|v^{\ast}_1|+1$. 
Since 
$|v^{\ast}_2|+ |v^{\ast}_3|+ |v^{\ast}_4|\geq 6$,
$D(G)\neq n-3$. 
In all other cases, 
$D(G)=|v^{\ast}_i|$, 
for a $1\leq i \leq 4$. 
On the other hand, the sum of the sizes of all three vertices of $G^{\ast}$ is greater than $3$.
This means that 
$D(G) \neq n-3$. 
In (g.2), (g.3) and (g.4), since 
$D(G)=\max \{|v^{\ast}_1|, |v^{\ast}_2|, |v^{\ast}_3|, |v^{\ast}_4|\}$, 
the result is immediate by Observation \ref{obser}. 
\end{proof}

\begin{theorem} 
Let 
$G$ 
be a graph of order 
$n\geq 5$. 
Let 
{\rsfs F} 
be the set of all graphs
except graphs $G$  with the property that 
${\rm dim}(G_c)=n-4$, 
${\rm diam}(G_c) \in \{2, 3\}$ 
and 
$5\leq |V(G_{c}^{\ast})| \leq 9$. 
If 
$G \in${\rsfs F},
then 
$D(G)=n-3$ 
if and only if 
$G$ 
is one of the following: 
\begin{itemize}
\begin{small}
\begin{multicols}{2} 
\item[(1)]
$P_5$
\item[(3)] 
$K_{4, 4}$
\item[(5)] 
$K_{3}+\overline{K_t}$, $t\geq 3$
\item[(7)] 
$K_{2}+(K_t \cup K_{1})$, $t\geq 2$
\item[(9)] 
$K_{t, 3}$, $t\geq 4$
\item[(11)] 
$K_t + \overline{K_{3}}$, $t\geq 3$
\item[(13)]
$K_{t}+(K_{2}\cup K_{1})$, $t\geq 2$ 
\item[(15)] 
$K_{1, 2, t}$, $t\geq 3$
\item[(17)] 
$K_2 + K_{2, 2}$
\item[(19)] 
$K_{1, 3, 3}$
\item[(21)] 
$K_{2, 2, 2}$ 
\item[(23)] 
$\overline{K_2}+(K_1 \cup K_t )$, $t\geq 2$
\item[(25)] 
$\overline{K_2} + 2K_{2}$
\item[(27)] 
$\overline{K_t}+(K_1 \cup K_2)$, $t\geq 2$
\item[(29)] 
$K_2 + 2K_{2}$
\item[(31)] 
$K_1+(K_1 \cup K_{1,t})$, $t\geq 2$ 
\item[(33)]
$K_1 + P_4$ 
\item[(35)] 
$K_1 + (K_1 \cup 2K_{2})$
\item[(37)] 
$K_1 + (K_1 \cup K_{2, 2})$  
\item[(39)] 
$P_4 [K_1, K_t, K_1, K_1]$, $t\geq 2$ 
\item[(41)] 
$P_4 [K_1, \overline{K_t}, K_1, K_1]$, $t\geq 2$ 
\item[(2)]
$C_5^{'}$
\item[(4)] 
$2K_{4}$
\item[(6)] 
$\overline{K_{3}} \cup K_t$, $t\geq 3$
\item[(8)] 
$\overline{K_{2}}\cup K_{t,1}$, $t\geq 2$
\item[(10)] 
$K_{t}\cup K_3$, $t\geq 4$
\item[(12)] 
$\overline{K_t} \cup K_{3}$, $t\geq 3$
\item[(14)]
$\overline{K_t}\cup K_{2,1}$, $t\geq 2$ 
\item[(16)] 
$K_{1}\cup K_{2} \cup K_t  $, $t\geq 3$
\item[(18)] 
$\overline{K_2} \cup 2K_{2}$
\item[(20)] 
$2K_{3}\cup K_1$
\item[(22)] 
$3K_{2}$ 
\item[(24)] 
$K_{2} \cup K_{t,1}$, $t\geq 2$
\item[(26)] 
$K_2 \cup K_{2, 2}$
\item[(28)] 
$K_t \cup K_{2,1}$, $t\geq 2$
\item[(30)] 
$2K_1 \cup K_{2, 2}$
\item[(32)] 
$K_1 \cup (K_1 + (K_t \cup K_1))$, $t\geq 2$ 
\item[(34)]
$K_1 \cup P_4$ 
\item[(36)] 
$K_1 \cup (K_1 + K_{2, 2})$
\item[(38)] 
$K_1 \cup (K_1 + 2K_{2})$
\item[(40)] 
$P_4 [\overline{K_t}, K_1, K_1, K_1]$, $t\geq 2$ 
\item[(42)] 
$P_4 [K_t, K_1, K_1, K_1]$, $t\geq 2$ 
\end{multicols}
\end{small}
\end{itemize}
\end{theorem} 
\begin{proof}
Assume first that 
$G$ 
is a connected graph with $D(G)=n-3$. Theorem \ref{main}, \cite[Theorem 3]{0} and Remark \ref{r1} conclude that 
${\rm dim}(G) \in \{n-2, n-3, n-4\}$. 
If ${\rm dim}(G) = n-2$, 
then let $\ell = 3$ in Lemma \ref{lemn-2}. Hence, $G$ is one of the graphs described in (3), (5), (7), (9), (11) or (13).  
Let ${\rm dim}(G) = n-3$. 
Thus 
${\rm diam}(G)\in \{2, 3\}$. 
If
${\rm diam}(G)=2$, 
then \cite[Theorem 1]{Jannesari} and Lemmas \ref{G1}, \ref{G2}, \ref{G3}, \ref{G5}, \ref{G6}, \ref{Gal} 
conclude the graphs in (15), (17), (19), (21), (23),  (25), (27), (29), (31), (33) and (2). 
Suppose ${\rm dim}(G) = n-4$. 
Then 
${\rm diam}(G)\in \{2, 3, 4\}$. 
If
${\rm diam}(G)=2$,  
then by \cite[Theorem 4.3]{dim n-4},
$4 \leq |V(G^{\ast})| \leq 9$.
If 
$|V(G^{\ast})|=4$, 
then by \cite[Theorem 4.1]{dim n-4} and Lemmas \ref{dim n-4 lem1}, \ref{dim n-4 lem2} 
we have the graphs in (35) and (37). 
If
${\rm diam}(G)=3$,  
then by \cite[Theorem 3.4]{dim n-4},
$4 \leq |V(G^{\ast})| \leq 7$.
If 
$|V(G^{\ast})|=4$, 
then by \cite[Theorem 3.1]{dim n-4} and Lemma \ref{dim n-4 lem3}, 
there is no graph $G$ with $D(G)=n-3$.
If
${\rm diam}(G)\in \{3, 4\}$ 
and 
${\rm dim}(G)=n-{\rm diam}(G)$, 
then the graphs in (1), (39), (40), (41) and (42) follow directly from \cite[Theorem 2.14]{Hernando} and Lemma \ref{Hernando-Lem}. 
If
$G$ 
is a disconnected graph, then by Lemma \ref{discon}, $G$ has the form described in (4), (6), $\ldots ,$ (38) as the complements of the graphs in (3), (5), $\ldots ,$ (37), respectively. 
(Note that the complement of graphs in (1), (39) and (41) are the ones in (2), (40) and (42), respectively, and all of them are connected).
\end{proof} 

\section{Conclusions and Future Research} 
This paper explores the relationship between the metric dimension and the distinguishing number.  
We show that for connected graphs, each resolving set results in a distinguishing coloring.
This connection facilitates faster resolution of certain problems in both concepts. For instance, certain graphs with large distinguishing number are identified here with the aid of this connection. Determining each resolving set, yields a distinguishing coloring that is not necessarily with the minimum number of colors. 
Therefore, graphs $G$ with 
$D(G)={\rm dim}(G)+1$
have significant value due to the fact that solving the metric dimension problem in these graphs and identifying the minimum resolving set will result in a distinguishing coloring that uses the fewest possible colors and the distinguishing vertices problem is completely solved. From our findings, obtaining graphs $G$ of order $n$ with
$D(G)={\rm dim}(G)+1 \in \{n-1, n-2\}$
is straightforward. Finding graphs with 
$D(G)={\rm dim}(G)+1=k$ 
can be an interesting task when classifications for graphs with
$D(G)=k$ and ${\rm dim}(G)=k-1$ are not explicit.  
It is also important to study connected graphs $G$ of order $n$ with the property that 
${\rm dim}(G)=n-4$, 
${\rm diam}(G) \in \{2, 3\}$ 
and 
$5\leq |V(G^{\ast})| \leq 9$ 
because it deals with two characterization problems of graphs with 
${\rm dim}(G)=n-4$ 
and 
$D(G)=n-3$.

\end{document}